\documentclass[a4paper]{amsart}

\usepackage{xcolor}

\def\PP{{\mathbb{P}}}

\def\QQ{{\mathbb{Q}}}

\newcommand{\oo}{{\mathcal O}}

\newcommand{\spec}{{\mathrm{Spec \ }}}

\newcommand{\m}{{\mathfrak{m}}}

\newcommand{\supp}{{\mathrm{supp}}}

\newcommand{\tP}{{\widetilde{P}}}

\newcommand{\ch}{{\rm char}}

\newtheorem{thm}{Theorem}[section]
\newtheorem{cor}[thm]{Corollary}
\newtheorem{Corollary-Definition}[thm]{Corollary-Definition}
\newtheorem{prop-def}[thm]{Proposition-Definition}

\newtheorem{lem}[thm]{Lemma}

\newtheorem{conj}[thm]{Conjecture}

\theoremstyle{definition}
\newtheorem{defn}[thm]{Definition}

\newtheorem{rem}[thm]{Remark}

\newtheorem{Proposition-Definition}[thm]{Proposition-Definition}

\begin{document}

\title[Vanishings]{Kawamata-Viehweg and Lipman's vanishings \\ for  sequences of \\ blow-ups at regular closed points}

\thanks{Mathematical Subject Classification 2020: 14B05,14E18, 14J17\\
Key words: singularities in arbitrary characteristic, Kawamata-Viehweg Vanishing Theorem \\
The author is partially supported by Grant-In-Aid (c) 22K03428 of JSPS 
and  by the Research Institute for Mathematical Sciences, an International Joint Usage/Research Center located in Kyoto University.
 }

\author{Shihoko Ishii }

\begin{abstract} The birational Kawamata-Viehweg vanishing theorem (KVV for short) is known to hold in considerable generality in equal characteristic zero, but fails in general in positive characteristic. We study a vanishing conjecture of Lipman for sequences of blow-ups at closed points of the spectrum of a regular local ring. We prove the conjecture when the ring has dimension at most three. We also establish a relative Kawamata-Viehweg vanishing theorem for exceptional nef \(\mathbf Q\)-divisors when all blow-up centers are rational over the residue field and the positive residue characteristic is at least the dimension of the ring minus one. Both results allow mixed characteristic.
\end{abstract}

\maketitle

\section{Introduction}
\noindent
Let $R$ be a regular local ring, and put $A_0:=\operatorname{Spec}R$.
Set $N:=\dim A_0$. Consider a sequence of blow-ups at closed points
$$
A_n\xrightarrow{\varphi_n}A_{n-1}\longrightarrow\cdots
\longrightarrow A_1\xrightarrow{\varphi_1}A_0,
\qquad
\varphi=\varphi_1\circ\cdots\circ\varphi_n.
$$
For each $i$, let $P_i\in A_{i-1}$ be the center of the blow-up $\varphi_i:A_i\to A_{i-1}$, and $E_i\subset A_i$ its exceptional divisor.
The proper transform of $E_i$ on $A_n$ is also denoted by $E_i$.
For each $i$, choose a canonical divisor $K_{A_i}$ compatibly, starting with the zero
divisor on $A_0$ and using the canonical
divisor formula for a blow-up at a closed point. 

\begin{defn}[Lipman]  Let $I\subset R$ be an ideal.
We say that $I$ is \emph{finitely supported} 
if there is a sequence as above such that
$$
I\mathcal{O}_{A_n}=\mathcal{O}_{A_n}(-D),
\qquad D=\sum a_iE_i,\qquad a_i\geq 0.
$$

\end{defn}

We consider the following conjecture:
\begin{conj}[Lipman's Vanishing Conjecture]
Let $I\subset R$ be a finitely supported ideal, and let $D$ be as above. 
Then, the following holds: 
$$R^q\varphi_*\mathcal{O}_{A_n}(K_{A_n}-D)=0
\qquad\text{for }q>0.
$$
\end{conj}

If  $R$ is  a $\QQ$-algebra,  the conjecture holds by Kawamata-Viehweg Vanishing theorem, as 
$-D$ is $\varphi$-nef ( \cite{mu}).

In positive or mixed characteristic, however,
 it is not obvious,
because KVV fails in general.

In \cite{lip}, Lipman asserted that his Theorem 2.1 proves the conjecture, and 
based on the vanishing in the conjecture he studied properties of finitely supported ideals.
However, his Theorem 2.1 is incorrect (see his comment \cite{lip2}).
This does not, however, disprove the conjecture. 
In this paper we will verify the conjecture for the following cases.
In both cases, mixed characteristic is allowed.

\begin{enumerate}
\item $N:=\dim R\leq 3$,
\item 
The residue field $k$ of $R$  
 satisfies $\dim R-1\leq p:=\ch k$, and
each blow-up center $P_i$ is a $k$-rational point.

\end{enumerate}

\noindent
{\bf Acknowledgement.} The author is grateful to Professor Joseph Lipman for his helpful correspondence and constructive comments. 
In particular, his suggestions led to an improvement of Theorem \ref{kvv}.

\vskip.5truecm
\section{The conjecture when $\dim R\leq 3$}   
\noindent
In this section we prove the conjecture when $\dim R\leq 3$.
First we give a definition.

\begin{defn}
Let $\varphi:Y\longrightarrow X$ be a proper birational morphism from
a regular scheme $Y$ to a normal affine scheme $X=\operatorname{Spec}B$.
Denote the dualizing sheaf of $Y$ by $\omega_Y$.
We say that \emph{Grauert--Riemenschneider vanishing}, or
\emph{G--R vanishing}, holds for $\varphi$ if
$$
R^q\varphi_*\omega_Y=0\qquad\text{for }q>0.
$$
\end{defn}

\begin{lem}\label{dim2}
Let $\varphi:Y\longrightarrow X$ be a resolution of singularities of normal scheme
$X=\operatorname{Spec}B$, where $\dim X=2$.
Then G--R vanishing holds for $\varphi$.
\end{lem}

\begin{proof}
This lemma follows from \cite[Theorem (2.3), (2.4)]{lip78}.

\end{proof}

\begin{lem}\label{g-r}
Let $R$ be a regular local ring, and let
$\varphi:A_n\longrightarrow A_0=\operatorname{Spec}R$ be a composition
of blow-ups at closed points. Then G--R vanishing holds for $\varphi$
in every dimension $N\geq 1$.
\end{lem}

\begin{proof}
It is known that regular implies rational (see, for example, \cite[Section 4]{lt})
Therefore, \[
\mathcal O_{A_0}\xrightarrow{\ \sim\ }
\mathbf R\varphi_*\mathcal O_{A_n}
\]
By \cite[Lemma 4.2]{lipCM} we  obtain 
\[
\mathbf R\varphi_*\omega_{A_n}\simeq\omega_{A_0}.
\]

\end{proof}

\begin{thm}\label{dim3}
Lipman's vanishing conjecture holds if $\dim R\leq3$.
\end{thm}

\begin{proof} If $\dim R=1$, every blow-up in the sequence is an isomorphism,
so the assertion is immediate.
When \(\dim R=2\), the argument of  \cite[Theorem 10.4]{kollar}, 
applied with the \(\varphi\)-nef divisor \(-D\), proves the conjecture. 
The excellence assumption imposed in \cite[\S 10.1]{kollar} is not needed for this argument.
 Indeed, the construction in \cite[Lemma 10.2]{kollar} and the resulting negative definiteness of the exceptional intersection
matrix do not require excellence. The remaining ingredients are adjunction,
Serre duality for proper curves over the residue field, and the theorem on formal
functions. Taking $L = \omega_Y$ and $N = \Delta_Y = 0$ gives $R^1\varphi_*\omega_Y = 0$. The higher direct
images in degrees greater than one vanish because the fibers have dimension at
most one.

Assume that $\dim A_0=3$. Put
$$
L=\mathcal{O}(-D)=I\mathcal{O}_{A_n},
\qquad I=(f_1,\ldots,f_r)\subset R.
$$
Let $\mathfrak{m}$ be the maximal ideal of $R$, $k:=R/\m$, and set
$$
R':=R[T_1,\ldots,T_r]_{\mathfrak{m}R[T_1,\ldots,T_r]}.
$$
The ring $R'$ has dimension $3$ and is faithfully flat over $R$.
If $\mathfrak{m}'\subset R'$ denotes its maximal ideal, then
$$
R'/\mathfrak{m}'=k(T_1,\ldots,T_r).
$$

Set
$$
A_n':=A_n\times_R R',\qquad L':=L|_{A_n'},
$$
where $L'$ is the pullback of $L$ to $A_n'$. Thus
$$
L'=\mathcal{O}_{A_n'}(-D'),
$$
where $D'$ is exceptional over $\operatorname{Spec}R'$.
Write $\varphi'$ for the morphism induced by $\varphi$ after this
base change.

Since $L$ is invertible on $A_n$, every point $P\in A_n$ has an
affine open neighborhood $U\subset A_n$ on which there exists
$h\in\Gamma(U, \mathcal{O}_U)$ such that
$$
L|_U=h\mathcal{O}_U.
$$
We can therefore write
$$
f_j|_U=h g_j,\qquad
 g_j\in\Gamma(U,\mathcal{O}_{A_n}).
$$
Since $h\in(f_1,\ldots,f_r)\mathcal{O}_U$, we can write
$$
h=\sum_{j=1}^r b_jf_j
=h\sum_{j=1}^r b_jg_j.
$$
Working in an integral domain and canceling $h$, we obtain
$$
1=\sum_{j=1}^r b_jg_j.
$$
Thus some $g_j$ is a unit in the local ring at $P\in U$.
After replacing $U$ by a sufficiently small affine open neighborhood
of $P$, we have
$$
g_j\in\Gamma(U,\mathcal{O}_{A_n})^*.
$$
On the open subset $U'\subset A_n'$ corresponding to $U$, we have
$$
L'|_{U'}=h\mathcal{O}_{U'}.
$$

Define
$$
F:=\sum_{j=1}^r T_jf_j
\in R[T_1,\ldots,T_r]
\hookrightarrow
R[T_1,\ldots,T_r]_{\mathfrak{m}R[T_1,\ldots,T_r]}=:R'.
$$
Since
$
F\in IR',
$
$F$ defines a global section of $L'$. Locally on $U'$, we have
$$
F|_{U'}=h\sum_{j=1}^r T_jg_j.
$$
The hypersurface of $A'_n$ defined by $\sum_{j=1}^r T_jg_j$ is regular on $U'$.
To see this, suppose that $g_1$ is a unit on $U\ni P$, and put
$$
B:=\Gamma(U,\mathcal{O}_{A_n}).
$$
Then
$$
\operatorname{Spec}\left(
B[T_1,\ldots,T_r]\Big/\Bigl(\sum_{j=1}^r T_jg_j\Bigr)
\right)
\simeq\operatorname{Spec}B[T_2,\ldots,T_r],
$$
which is regular.
As the local ring of a point on $U'$ is a localization of this regular ring, therefore it is  a regular local ring.
We have
$$
\operatorname{div}_{A_n'}F=D'+H\sim 0,
$$
where $D'$ is locally defined by $h$ and $H$ is locally defined by
$\sum_{j=1}^r T_jg_j$. Hence
$$
\mathcal{O}_{A_n'}(H)
\simeq\mathcal{O}_{A_n'}(-D')=L'.
$$
Set
$$
S:=\operatorname{Spec}(R'/(F)),
\qquad\rho:H\longrightarrow S,
$$
where $\rho$ is the restricted morphism. It is birational and is an
isomorphism outside the closed point. The scheme $S$ is a hypersurface
that is regular away from the closed point. Thus $S$ satisfies $S_2$ and $R_1$, and
is a normal surface.

Since $L'\simeq\mathcal{O}_{A_n'}(H)$, we have the exact sequence
$$
0\longrightarrow\omega_{A_n'}
\longrightarrow\omega_{A_n'}\otimes L'
\longrightarrow\omega_H\longrightarrow 0.
$$
For $q>0$, the associated higher direct image sequence contains
$$
R^q\varphi'_*\omega_{A_n'}
\longrightarrow
R^q\varphi'_*\bigl(\omega_{A_n'}\otimes L'\bigr)
\longrightarrow R^q\rho_*\omega_H.
$$
The left-hand term vanishes by Lemma~\ref{g-r} and flat base change.
Since $\dim H=2$, the right-hand term vanishes by Lemma~\ref{dim2}.
It follows that
$$
R^q\varphi'_*\bigl(\omega_{A_n'}\otimes L'\bigr)=0
\qquad(q>0).
$$
By flat base change (\cite[III, Proposition 9.3]{ha}), tensoring with $R'$ commutes with the
higher direct images in question. 
Since $R'$ is faithfully flat
over $R$, the vanishing after tensoring with $R'$ implies
$$
R^q\varphi_*\bigl(\omega_{A_n}\otimes L\bigr)
=R^q\varphi_*\omega_{A_n}(-D)=0
\qquad(q>0).
$$
This proves the theorem.

\end{proof}

\begin{rem} The proof of the theorem depends on $\dim S=2$ so that a resolution $H\to S$ of any normal surface $S$ satisfies G-R vanishing theorem.
When $\dim S\geq 3$, it is well known that G-R vanishing theorem does not hold in general.
For example, a resolution of a cone over Raynaud surface (\cite{ray}) does not satisfy G-R vanishing.
However, as the variety is not Cohen-Macaulay, it does not arise as  $S$ in our discussion, since our $S$ is a hypersurface.
\end{rem}

\begin{rem} In dimension three, G--R vanishing holds for excellent Cohen--Macaulay schemes of klt type by \cite[Corollary 1.4]{bkr}.
 For a G--R vanishing theorem on three-dimensional excellent rational singularities, see \cite[Theorem A]{aj}. 
 However, our construction does not guarantee either type of singularity.    
    
    For example, let 
    $$R:=k[x_0,\ldots,x_d]_{(x_0,\ldots,x_d)},\quad I=\m^e, \quad e>d\geq 2,$$
     where $\m\subset R$ is the maximal ideal.
     Let  
    $A_0=\spec R$.
    Consider the blow-up $\varphi_1: A_1\to A_0$.
    Let $f_1,\ldots f_r$ be  all the degree $e$ monomials in a regular system of parameters,
    and define 
    $$F=\sum_i f_iT_i.$$
    In this case, $D=eE_1$, $D'=eE'_1$, and $H\to S$ is the blow-up of the localization of the affine cone $S$ at its vertex over the
    smooth hypersurface of degree $e$ in $\PP_{k'}^d$,
    where $k'=k(T_1,\ldots, T_r)$. 
    Therefore, $S$ is not of klt type or rational.
    However, we can verify that this $H\to S$ still satisfies G-R vanishing.

\end{rem}

\vskip1truecm
\section{Vanishing for blow-ups at rational closed points}

\noindent
In this section, we prove the following:

\begin{thm}[KVV for a sequence of blow-ups at closed points]\label{kvv} 
Let $R$ be a regular local ring with  residue field $k$  of characteristic $p>0$.
Set $A_0:=\spec R$ and $N:=\dim R\geq 2$, and assume that $p\geq N-1$.
Let
$$
A_n\xrightarrow{\varphi_n}A_{n-1}\xrightarrow{\varphi_{n-1}}\cdots
\xrightarrow{\varphi_2}A_1\xrightarrow{\varphi_1}A_0
$$
be a sequence in which each $\varphi_i$ is the blow-up at a $k$-rational closed point $P_i\in A_{i-1}$, and set
$$
\varphi:=\varphi_1\circ\cdots\circ\varphi_n:A_n\longrightarrow A_0.
$$
For each $i$, let $E_i\subset A_i$ be the exceptional divisor of $\varphi_i$, and denote its strict transform on $A_j $ for $j>i$ by the same symbol $E_i$.
Let
$$
L=\sum_{i=1}^n a_iE_i
$$
be a $\varphi$-nef $\mathbb Q$-divisor on $A_n$.
Then
\begin{equation}\label{vanish}
R^q\varphi_*\mathcal O_{A_n}\bigl(K_{A_n}+\lceil L\rceil\bigr)=0
\qquad\text{for every }q>0,
\end{equation}
where 
$$
\lceil L\rceil:=\sum_{i=1}^n\lceil a_i\rceil E_i.
$$
\end{thm}

We first prove the theorem under the additional assumptions that \(R\) is excellent and \(k\) is perfect. 
We then deduce the general case by faithfully flat base change.

\begin{lem}\label{on-exceptional} Retain the notation and hypotheses of Theorem \ref{kvv}, 
and assume in addition that \(k\) is perfect.
Fix $i\in\{1,\ldots,n\}$, and regard $E_i$ as its strict transform on $A_n$. Set
$$
B_i:=\left(\left.\sum_{j\neq i}E_j\right|_{E_i}\right)_{\mathrm{red}}.
$$
Let $L_i$ be an ample $\QQ$-divisor on $E_i$, and assume that
$$
\operatorname{Supp}\langle L_i\rangle
\subseteq
\operatorname{Supp}B_i,
\qquad
\langle L_i\rangle:=L_i-\lfloor L_i\rfloor.
$$
Then
$$
H^q\left(
E_i,
\oo_{E_i}\left(K_{E_i}+\lceil L_i\rceil\right)
\right)=0
$$
for every $q>0$.
\end{lem}

\begin{proof}
Put $d:=\dim E_i=N-1$. Since $A_n$ is obtained by successive blow-ups
at closed points, $E_i$ is a smooth projective variety, and $B_i$ is a
simple normal crossings divisor on $E_i$.
Although $R$ is not assumed to be essentially of finite type over a field,
each exceptional divisor $E_i$ is a smooth projective variety over $k$,
since it is obtained from $\mathbb P_k^{N-1}$ by successive blow-ups
at $k$-rational closed points.
All log smooth pairs and their liftings considered below are over $k$
and $W_2(k)$, respectively; no lifting of the ambient schemes $A_j$
is required.

We first show that the log smooth pair $(E_i,B_i)$ lifts to $W_2(k)$.
At the stage when $E_i$ is created, it is isomorphic to $\PP_k^d$, and
its intersections with the earlier exceptional divisors form a union
of coordinate hyperplanes. Hence this initial log smooth pair lifts to
the corresponding pair over $W_2(k)$.

We now proceed through the subsequent blow-ups. If the next center does
not lie on the current strict transform of $E_i$, the pair on $E_i$
remains unchanged. If the center lies on $E_i$, then the induced
morphism on $E_i$ is the blow-up at that closed point.

Assume inductively that the current log smooth pair has been lifted
over $W_2(k)$. We use the same notation $(E_i,B_i)$ and
$(\widetilde E_i,\widetilde B_i)$ for the current pair and its lifting.
Let $P_j\in E_i$, with $j>i$, be the next blow-up center lying on $E_i$.
Let $B_{i,1},\ldots,B_{i,r}$ be the irreducible components of $B_i$,
and denote their liftings by $\widetilde B_{i,\alpha}$.
After renumbering these components, we may assume that
$P_j\in B_{i,\alpha}$ precisely when $1\leq\alpha\leq s$.
Consider the locally closed subscheme
\[
\widetilde Z:=
\left(\bigcap_{\alpha=1}^{s}\widetilde B_{i,\alpha}\right)
\setminus
\left(\bigcup_{\alpha=s+1}^{r}\widetilde B_{i,\alpha}\right)
\subseteq\widetilde E_i,
\]
where the intersection is scheme-theoretic, and an empty intersection
is understood to be $\widetilde E_i$.
Since $\widetilde B_i$ is a relative simple normal crossings divisor,
$\widetilde Z$ is smooth over $W_2(k)$.
The $k$-rational point $P_j$ lies in its special fiber.
By the infinitesimal lifting property of smooth morphisms, it lifts
to a $W_2(k)$-valued point of $\widetilde Z$.
Denote the image of this section by $\tP_j$; thus
$\tP_j\simeq\operatorname{Spec}W_2(k)$ and
$(\tP_j)_k=P_j$.
By construction,
\[
\begin{cases}
\tP_j\subseteq\widetilde B_{i,\alpha}
  &\text{if }P_j\in B_{i,\alpha},\\[2pt]
\tP_j\cap\widetilde B_{i,\alpha}=\emptyset
  &\text{if }P_j\notin B_{i,\alpha}.
\end{cases}
\]
The inclusion in the first line is scheme-theoretic.
This is the precise sense in which the incidence relations between
the center and the boundary components are preserved.

Therefore, blowing up $\widetilde{E}_i$ along $\widetilde{P}_j$
produces a smooth scheme over $W_2(k)$, and the strict transforms
of the boundary components, together with the inverse image of
$\widetilde{P}_j$, form a relative simple normal crossings divisor.
The resulting pair specializes to the pair obtained by blowing up
$E_i$ at $P_j$ and taking the new boundary induced by the ambient
exceptional divisors. Continuing inductively, we obtain a log smooth lifting
$(\widetilde E_i,\widetilde B_i)$ of the final pair $(E_i,B_i)$
over $W_2(k)$.

Let
$$
S_i:=\operatorname{Supp}\langle L_i\rangle.
$$
By assumption, $S_i$ is a union of components of $B_i$. Hence the
corresponding union of components of $\widetilde B_i$ gives a log
smooth lifting of $(E_i,S_i)$ to $W_2(k)$.

Applying \cite[Theorem~5.1(ii)]{m-o} to the ample $\QQ$-divisor $L_i$,
we obtain
$$
H^a\left(
E_i,
\Omega_{E_i}^b(\log S_i)\otimes
\oo_{E_i}\left(-\lceil L_i\rceil\right)
\right)=0
$$
whenever
$$
a+b<\min\{d,p\}.
$$
Since $p\geq N-1=d$, taking $b=0$ gives
$$
H^a\left(
E_i,
\oo_{E_i}\left(-\lceil L_i\rceil\right)
\right)=0
\qquad
\text{for every }0\leq a<d.
$$
By Serre duality, for every $1\leq q\leq d$,
$$
H^q\left(
E_i,
\oo_{E_i}\left(K_{E_i}+\lceil L_i\rceil\right)
\right)^\vee
\simeq
H^{d-q}\left(
E_i,
\oo_{E_i}\left(-\lceil L_i\rceil\right)
\right)=0.
$$
The assertion follows.
\end{proof}


\vskip.5truecm
\noindent
{
\it Proof of Theorem \ref{kvv}.}
Assuming that $R$ is excellent and $k$ is perfect, we prove the vanishing by adapting the proof of  \cite[Proposition~8]{arv}.
In this setting, the following hypotheses in that proposition are satisfied:
\begin{enumerate}
\item[$\bullet$] $R$ is an excellent local domain;
\item[$\bullet$] $(A_n,\operatorname{Exc}(\varphi))$ is dlt, where $\operatorname{Exc}(\varphi)$ denotes the reduced exceptional divisor;
\item[$\bullet$] each $E_i$ is a $\QQ$-Cartier divisor; and
\item[$\bullet$] every reflexive rank-one sheaf on $A_n$ is $S_3$.
\end{enumerate}
Indeed, $A_n$ is regular, and its reduced exceptional divisor is a simple normal crossings divisor.
Arvidsson also assumes that the induced morphism $E_i\to\spec k$ satisfies KVV for every $i$.
For our purpose, however, it is enough to verify the particular KVV statements that occur in the proof of \cite[Proposition~8]{arv}; Lemma~\ref{on-exceptional} gives precisely these statements.
We recall the relevant part of that proof to make this point explicit.

We put $B=0$ in the notation of \cite{arv}.
First assume that $L$ is a $\varphi$-ample exceptional $\QQ$-divisor on $A_n$.

Put
$$
\Delta:=\lceil L\rceil-L.
$$
Then $\Delta$ is an effective $\varphi$-exceptional $\QQ$-divisor whose coefficients are strictly smaller than $1$.

Let $\Delta_0:=\Delta$ and $D_0:=K_{A_n}+\Delta_0+L=K_{A_n}+\lceil L\rceil$.
We prove
$$
R^q\varphi_*\oo_{A_n}(D_0)=0
\qquad\text{for }q>0.
$$
Choose an effective integral $\varphi$-exceptional divisor $G$ such that
$$
\supp G=\bigcup_{i=1}^n E_i
\qquad\text{and}\qquad
-G\text{ is }\varphi\text{-ample}.
$$
Take $\mu_0>0$ so that every coefficient of $\Delta_0+\mu_0G$ is at most $1$ and at least one coefficient is equal to $1$.
Choose $i_0$ such that the coefficient of $E_{i_0}$ in $\Delta_0+\mu_0G$ is $1$, and define
$$
\Delta_1:=\Delta_0+\mu_0G-E_{i_0},
\qquad
D_1:=D_0-E_{i_0}.
$$
Then $E_{i_0}$ is not contained in $\supp\Delta_1$.
Since
$$
D_0=K_{A_n}+E_{i_0}+\Delta_1+L-\mu_0G,
$$
we obtain the exact sequence
$$
0\to \oo_{A_n}(D_1)\to \oo_{A_n}(D_0)
\to \oo_{E_{i_0}}(D_0|_{E_{i_0}})\to 0.
$$
Choosing $K_{E_{i_0}}$ compatibly with adjunction, we have
$$
D_0|_{E_{i_0}}
=
K_{E_{i_0}}+\Delta_1|_{E_{i_0}}+L_1,
$$
where
$$
L_1:=(L-\mu_0G)|_{E_{i_0}}
$$
is an ample $\QQ$-divisor on $E_{i_0}$, and
$$
\supp(\Delta_1|_{E_{i_0}})\subseteq\supp B_{i_0}.
$$
Choose a sufficiently small rational number $0<\varepsilon_0\ll1$ such that
$$
L_1+\varepsilon_0\Delta_1|_{E_{i_0}}
$$
is ample. Since $D_0|_{E_{i_0}}-K_{E_{i_0}}$ is integral and
$$
L_1+\varepsilon_0\Delta_1|_{E_{i_0}}
=
D_0|_{E_{i_0}}-K_{E_{i_0}}
-(1-\varepsilon_0)\Delta_1|_{E_{i_0}},
$$
we have
$$
\left\lceil L_1+\varepsilon_0\Delta_1|_{E_{i_0}}\right\rceil
=
D_0|_{E_{i_0}}-K_{E_{i_0}},
$$
and
$$
\supp\left\langle L_1+\varepsilon_0\Delta_1|_{E_{i_0}}\right\rangle
\subseteq
\supp B_{i_0}.
$$
Lemma~\ref{on-exceptional} therefore gives
$$
H^q\left(E_{i_0},\oo_{E_{i_0}}(D_0|_{E_{i_0}})\right)=0
\qquad(q>0).
$$

We now repeat the same procedure. For each $j\geq1$, choose $\mu_j\geq0$ and an index $i_j$ so that every coefficient of $\Delta_j+\mu_jG$ is at most $1$ and the coefficient of $E_{i_j}$ is equal to $1$, and set
$$
\Delta_{j+1}:=\Delta_j+\mu_jG-E_{i_j},
\qquad
D_{j+1}:=D_j-E_{i_j}.
$$
Together with the first step, we obtain exact sequences, for every $j\geq0$,
$$
0\to \oo_{A_n}(D_{j+1})\to \oo_{A_n}(D_j)
\to \oo_{E_{i_j}}(D_j|_{E_{i_j}})\to 0,
$$
and
$$
D_{j+1}-(K_{A_n}+\Delta_{j+1})
=
D_j-(K_{A_n}+\Delta_j)-\mu_jG.
$$
Thus
$$
D_j=K_{A_n}+\Delta_j+L-\lambda_jG,
$$
where $\lambda_0=0$ and $\lambda_{j+1}:=\lambda_j+\mu_j$.
Choosing $K_{E_{i_j}}$ compatibly with adjunction, at the $j$-th step we have
$$
D_j|_{E_{i_j}}
=
K_{E_{i_j}}+\Delta_{j+1}|_{E_{i_j}}
+(L-\lambda_{j+1}G)|_{E_{i_j}}.
$$
The last term is ample, and
$$
\supp(\Delta_{j+1}|_{E_{i_j}})\subseteq\supp B_{i_j}.
$$
Exactly as above, after adding a sufficiently small positive multiple of $\Delta_{j+1}|_{E_{i_j}}$ to the ample divisor $(L-\lambda_{j+1}G)|_{E_{i_j}}$, Lemma~\ref{on-exceptional} yields
$$
H^q\left(
E_{i_j},
\oo_{E_{i_j}}(D_j|_{E_{i_j}})
\right)=0
\qquad(q>0).
$$
The final part of the proof of \cite[Proposition~8]{arv} applies without change and gives
$$
R^q\varphi_*\oo_{A_n}(D_j)=0
\qquad(q>0)
$$
for $j\gg0$.
The exact sequences above then give, by descending induction,
$$
R^q\varphi_*\oo_{A_n}(D_j)=0
\qquad(q>0)
$$
for every $j$. In particular, for a $\varphi$-ample exceptional $\QQ$-divisor $L$,
$$
R^q\varphi_*\oo_{A_n}(K_{A_n}+\lceil L\rceil)
=
R^q\varphi_*\oo_{A_n}(D_0)=0
\qquad(q>0).
$$

Now let $L$ be a $\varphi$-nef $\QQ$-divisor supported on the exceptional divisor.
For a sufficiently small rational number $\varepsilon>0$, the divisor
$L-\varepsilon G$ is $\varphi$-ample and
$$
\lceil L-\varepsilon G\rceil=\lceil L\rceil.
$$
Therefore, the desired vanishing for a $\varphi$-nef divisor reduces to the $\varphi$-ample case proved above.
This completes the proof for an excellent $R$ with the perfect residue field $k$.

We now treat the general case.
Let $K$ be an algebraic closure of $k=R/\mathfrak m$.
By \cite[(6.8.2)]{ega}, there exists a flat local homomorphism
$R\to Q_0$ such that $\mathfrak mQ_0$ is the maximal ideal
of $Q_0$ and
\[
Q_0/\mathfrak mQ_0\simeq K.
\]
Set
\[
Q:=\varprojlim_r Q_0/\mathfrak m^rQ_0.
\]
Since $\mathfrak mQ_0$ is finitely generated and
$Q_0/\mathfrak mQ_0$ is Noetherian, $Q$ is a Noetherian
complete local ring by \cite[Tag05GH]{sta}.
 Moreover,
\[
Q/\mathfrak m^rQ\simeq Q_0/\mathfrak m^rQ_0
\qquad(r\ge1);
\]
see \cite[Tag05GG]{sta}.
Each of these quotients is flat over $R/\mathfrak m^r$.
For every ideal $\mathfrak a\subset R$, the $Q$-module
$\mathfrak a\otimes_R Q$ is finitely generated and hence
$\mathfrak mQ$-adically separated by Krull's intersection theorem.
Thus $Q$ is $\mathfrak m$-adically ideal-separated as an $R$-module.
Applying $(5)\Rightarrow(1)$  in \cite[Theorem~22.3]{mat} with
$A=R$, $I=\mathfrak m$, and $M=Q$, we conclude that $R\to Q$ is flat.
Since this homomorphism is local, it is faithfully flat.
Let $x_1,\ldots,x_N$ be a regular system of parameters of $R$.
Their images form a regular sequence in $Q$, by flatness,
and generate its maximal ideal:
\[
\mathfrak mQ=(x_1,\ldots,x_N)Q.
\]
Thus $Q$ is a regular local ring of dimension $N$.
Being complete, it is excellent, and its residue field $K$
is perfect.

For each $i$, set
\[
A_i':=A_i\times_{\operatorname{Spec}R}\operatorname{Spec}Q.
\]
Since every center $P_i$ is $k$-rational, its base change is
\[
P_i':=P_i\times_{\operatorname{Spec}R}\operatorname{Spec}Q
\simeq\operatorname{Spec}K.
\]
Blow-ups commute with flat base change. Consequently,
\[
A_n'\longrightarrow A_{n-1}'\longrightarrow\cdots
\longrightarrow A_0'=\operatorname{Spec}Q
\]
is a sequence of blow-ups at $K$-rational closed points.
Let $\varphi':A'_n\to A'_0$ be the composite of these blow-ups.
Let $\rho_n:A_n'\to A_n$ be the projection, and put
\[
L':=\rho_n^*L=\sum_i a_iE_i'.
\]
Then $L'$ is $\varphi'$-nef and
\[
\lceil L'\rceil=\rho_n^*\lceil L\rceil.
\]
Choose canonical divisors compatibly, starting with the zero
divisor on each regular local base and using the canonical
divisor formula for a blow-up at a closed point. This gives
\[
K_{A_n'}=\rho_n^*K_{A_n},
\]
and hence
\[
\rho_n^*\mathcal O_{A_n}(K_{A_n}+\lceil L\rceil)
\simeq
\mathcal O_{A_n'}(K_{A_n'}+\lceil L'\rceil).
\]
The case already proved for excellent regular local rings
with perfect residue fields therefore yields
\[
R^q\varphi'_*
\mathcal O_{A_n'}(K_{A_n'}+\lceil L'\rceil)=0
\qquad(q>0).
\]
Let $\rho:\operatorname{Spec}Q\to\operatorname{Spec}R$
be the induced morphism. Flat base change gives
\[
\rho^*R^q\varphi_*
\mathcal O_{A_n}(K_{A_n}+\lceil L\rceil)
\simeq
R^q\varphi'_*
\mathcal O_{A_n'}(K_{A_n'}+\lceil L'\rceil)
=0.
\]
Since $\rho$ is faithfully flat, we conclude that
\[
R^q\varphi_*\mathcal O_{A_n}(K_{A_n}+\lceil L\rceil)=0
\qquad(q>0).
\]
$\Box$

\begin{cor} Lipman's vanishing conjecture holds if the residue field $k$ of $R$ satisfies
  $\dim R-1\leq p:=\ch k$, and
each blow-up center $P_i$ is a $k$-rational point.

\end{cor}

\begin{proof}
Set
$$
D:=\sum_{i=1}^n a_iE_i.
$$
The natural map
$$
\varphi^*I\longrightarrow I\mathcal O_{A_n}=\mathcal O_{A_n}(-D)
$$
is surjective.
Hence $\mathcal O_{A_n}(-D)$ is generated relative to $A_0$, so $-D$ is $\varphi$-nef.
Applying Theorem~\ref{kvv} to the integral $\varphi$-nef divisor $L=-D$, we obtain
$$
R^q\varphi_*\mathcal O_{A_n}(K_{A_n}-D)=0
\qquad(q>0).
$$

\end{proof}




\vskip1truecm

\noindent Shihoko Ishii, \\ Graduate School of Mathematical Sciences, University of Tokyo,\\
3-8-1, Komaba, Meguro-ku, Tokyo, Japan.\\
shihoko@g.ecc.u-tokyo.ac.jp

\end{document}